\documentclass[11pt]{article}
\usepackage{amsthm}
\usepackage{amsmath,amsfonts,amssymb,graphics,verbatim,fullpage}
\usepackage[usenames,dvipsnames]{color}

\ifx\pdfoutput\undefined
\usepackage[dvips,bookmarks]{hyperref}
\else
\usepackage{hyperref}
\fi
\hypersetup{%
pdftitle={},
pdfauthor={},
pdfkeywords={},
bookmarks=true,%
bookmarksnumbered=true,
pdfstartview={FitH},
colorlinks=true,
citecolor=Plum,
linkcolor=RedOrange,
urlcolor=RawSienna,
bookmarksopen=true,%
bookmarksopenlevel=1,
unicode=true,%
breaklinks=true,%
}%

\def\te#1{\mathrm{e}^{#1}}
\def\td{\text{\rm d}}
\numberwithin{equation}{section}

\newtheorem{theorem}{Theorem}[section]
\newtheorem{lem}{Lemma}[section]

\newcommand{\eqdistr}{\stackrel{{\cal D}}{=}}

\theoremstyle{definition}
\newtheorem*{remark}{Remark}
\newtheorem{Ex}{Example}

\begin{document}
\title{A L\'{e}vy input model with additional state-dependent services}
\author{Zbigniew Palmowski\footnote{University of Wroc\l aw, pl.\ Grunwaldzki 2/4, 50-384 Wroc\l aw, Poland, E-mail: zpalma@math.uni.wroc.pl} \qquad Maria Vlasiou\footnote{Eurandom and Department of Mathematics and Computer Science; Eindhoven University of Technology; P.O.\ Box 513; 5600 MB Eindhoven; The Netherlands, E-mail: m.vlasiou@tue.nl}}
\maketitle

\begin{abstract}
We consider a queuing model with the workload evolving between consecutive i.i.d.\ exponential timers $\{e_q^{(i)}\}_{i=1,2,\ldots}$ according to a spectrally positive L\'{e}vy process $Y(t)$ which is reflected at $0$. When the exponential clock $e_q^{(i)}$ ends, the additional state-dependent service requirement modifies the workload  so that the latter is equal to $F_i(Y(e_q^{(i)}))$ at epoch $e^{(1)}_q+\cdots+e^{(i)}_q$ for some random nonnegative i.i.d.\ functionals $F_i$. In particular,  we focus on the case when $F_i(y)=(B_i-y)^+$, where $\{B_i\}_{i=1,2,\ldots}$ are i.i.d.\ nonnegative random variables. We analyse the steady-state workload distribution for this model.
\end{abstract}

\section{Introduction}

In this paper we focus on a particular queuing system with additional state-dependent services. There has been considerable previous work on queues with state-dependent service and arrival processes; see for example the survey by Dshalalow~\cite{dshalalow97} for several references. The model under consideration involves a reflected  L\'evy process connected to the evolution of the workload. Special cases of L\'evy processes are the compound Poisson process, the Brownian motion, linear drift processes, and independent sums of the above. The literature on queueing systems driven by L\'evy processes is rather limited; see e.g.\ Bekker et al.~\cite{bekker} for references.

Specifically, in this paper we consider a storage/workload model in which the workload evolves according to a reflected at zero spectrally positive L\'{e}vy process $Y(t)$. That is, let $X(t)$ be a spectrally positive L\'{e}vy process (a L\'{e}vy process with only positive jumps) modelling the input minus the output of the process and define $-\inf_{s\leqslant 0} X(s)=0$ and $X(0)=x\geqslant 0$. Then we have that $Y(t)=X(t)-\inf_{s\leqslant t} X(s)$ (where $Y(0)=x$ for some initial workload $x\geqslant 0$). In addition, at exponential times with intensity $q$, given by $\{e_q^{(i)}\}_{i=1,2,\ldots}$, the workload is ``reset'' to a certain level, depending on the workload level before the exponential clock ends. Specifically, at epoch $t=e^{(1)}_q+\cdots+e^{(i)}_q$ the workload $V(t)$ equals $F_i(V(t-))$ for some random nonnegative i.i.d.\  functionals $F_i$.

The main goal of our paper is to derive the stationary distribution of the workload $V(t)$ for the above-described queuing model. We first identify the stationary distribution of the workload at embedded exponential epochs and then extend this result to an arbitrary time by using renewal arguments. We also identify the tail behaviour of the steady-state workload.

This model unifies and extends several related models in various directions. First of all, if $X$ is a compound Poisson process and if $F_i$ is the identity function, then our model reduces to the workload process of the M/G/1 queue. Kella et al.~\cite{Kellatotalclearing} consider a  model with workload removal, which fits into our model by taking $F_i(x)=0$ and by letting the  spectrally positive L\'{e}vy process $X$ be a Brownian motion superposed with an independent compound Poisson component. The added generality allows one to analyse more elaborate mechanisms of workload control, where the exponential times can be seen as review times, during which the workload can be changed to a different level as desired. Allowing for general functions $F_i$ opens up the possibility of optimising such controls, although we do not consider this problem here. Instead of considering the classical compound Poisson input process and a linear output process for the queuing model, one can consider the case in which the L\'evy process is nondecreasing, i.e.\ a subordinator; see for example Bekker et al.~\cite{bekker} and Boxma et al.~\cite{boxma}.

Apart from the wish to unify and extend clearing models, we were also challenged by introducing a continuous-time analogue of the alternating service model considered in Vlasiou et al.~\cite{vlasiou07a,vlasiou05,vlasiou05b,vlasiou04,vlasiou07} that gave rise to the Lindley-type recursion $W\eqdistr (B-A-W)^+$. In particular, we focus on the case when $F_i(y)=(B_i-y)^+$, where $\{B_i\}_{i=1,2,\ldots}$ are i.i.d.\ nonnegative random variables. Hence, at exponential epochs the controlling mechanism leaves only a portion of the workload depending on the size of the workload just prior to the exponential timer. In particular, if $F_i(y)=(B_i-y)^+$, then this mechanism keeps the workload below a generic random size $B$, decreasing it when it is relatively large at the exponential epoch and increasing it when it is much smaller than $B$. This can be viewed as a continuous-time analogue of the above mentioned Lindley-type recursion. In the context of workload control mentioned above, this example can be interpreted as a mechanism that reflects existing storage with respect to an upper bound $B$.

Our results focus on qualitative and quantitative properties of the steady-state workload distribution. We first establish Harris recurrence for the Markov chain embedded at workload adjustment points, yielding the convergence of the workload processes to an invariant distribution. We derive an equation for the invariant distribution of the embedded chain, as well as the invariant distribution of the original process. We use this equation to obtain expressions for the invariant distribution for an example that generalises \cite{bekker,boxma,Kellatotalclearing}. We also investigate the tail behaviour of the steady-state distribution under both light-tailed and heavy-tailed assumptions. All these groups of results have the common theme that we rely on recently obtained results in the fluctuation theory of spectrally positive L\'evy processes.

The paper is organised as follows. In Section~\ref{s:Preliminaries} we introduce a few basic facts concerning spectrally positive L\'{e}vy processes. In Section~\ref{s:embedded} we consider the embedded workload process and derive a recursive equation for its stationary distribution. In Section~\ref{s:arbitrary} we determine the steady-state workload distribution. Later on, in Section~\ref{s:examples} we present some special cases. Finally, in Section \ref{tail} we focus on the tail behaviour of the steady-state workload.

\section{Preliminaries}\label{s:Preliminaries}

Throughout this paper we exclude the case of $X$ with monotone paths. Let the dual process of $X(t)$ be given by $\hat{X}(t)=-X(t)$. The process $\{\hat{X}(s), s\leqslant t\}$ is a spectrally negative L\'{e}vy process and has the same law as the time-reversed process $\{X((t-s)-)-X(t), s\leqslant t\}$. Following standard conventions, let $\underline{X}(t)=\inf_{s\leqslant t}X(s)$, $\overline{X}(t)=\sup_{s\leqslant t}X(s)$ and similarly $\underline{\hat{X}}(t)=\inf_{s\leqslant t}\hat{X}(s)$, and $\overline{\hat{X}}(t)=\sup_{s\leqslant t}\hat{X}(s)$. One can readily see that the processes $Y(t)=X(t)-\underline{X}(t)$ (for $Y(0)=0$) and  $\overline{X}(t)$ (where $X(0)=0$) have the same distribution; see e.g.\ Kyprianou~\cite[Lemma 3.5, p.\ 74]{Kbook}. Moreover,
$$
-\underline{X}(t)\eqdistr \overline{\hat{X}}(t),\qquad \overline{X}(t)\eqdistr -\underline{\hat{X}}(t).
$$

Since the jumps of $\hat{X}$ are all non-positive, the moment generating function $E[\te{\theta \hat{X}(t)}]$ exists for all $\theta \geqslant 0$ and is given by $E[\te{\theta \hat{X}(t)}] = \te{t\psi(\theta)}$ for some function $\psi(\theta)$ that is well defined at least on the positive half-axis where it is strictly convex with the property that $\lim_{\theta\to\infty}\psi(\theta)=+\infty$. Moreover, $\psi$ is strictly increasing on $[\Phi(0),\infty)$, where $\Phi(0)$ is the largest root of $\psi(\theta)=0$. We shall denote the right-inverse function of $\psi$ by $\Phi:[0,\infty)\to[\Phi(0),\infty)$.

Denote by $\sigma$ the Gaussian coefficient and by $\nu$ the L\'{e}vy measure of $\hat{X}$ (note that $\sigma$ is also a Gaussian coefficient of $X$ and that $\Pi(A)=\nu(-A)$ is a jump measure of $X$).  Throughout this paper we assume that the following (regularity) condition is satisfied:
\begin{equation}\label{eq:condW}
\sigma > 0 \quad \text{or}\quad \int_{-1}^0 x \nu(\td x) = \infty \quad \text{or}\quad \nu(\td x)<<\td x,
\end{equation}
where $<< \td x $ means absolutely continuity with respect to the Lebesgue measure.
Moreover, we assume that
\begin{equation}
P_x(\tau_0^-<\infty)=1,
\end{equation}
where
$$
\tau_0^-=\inf\{t\geqslant 0: X(t)\leqslant 0\}.
$$
Finally, $P_x$ denotes the probability measure $P$ under the condition that $X(0)=x$, and $E_x$ indicates the expectation with respect to $P_x$.

\subsection{Scale functions}\label{ss:scale}
For $q\geqslant0$, there exists a function $W^{(q)}: [0,\infty) \to [0,\infty)$, called the {\it $q$-scale function}, that is continuous and increasing with Laplace transform
\begin{equation}\label{eq:defW}
\int_0^\infty \te{-\theta y} W^{(q)} (y)  \td y = (\psi(\theta) - q)^{-1},\qquad\theta > \Phi(q).
\end{equation}
The domain of $W^{(q)}$ is extended to the entire real axis by setting $W^{(q)}(y)=0$ for $y<0$. We mention here some properties of the function $W^{(q)}$ that have been obtained in the literature which we will need later on.

On $(0,\infty)$ the function $y\mapsto W^{(q)}(y)$ is right- and left-differentiable and, as shown in \cite{lam}, under the condition \eqref{eq:condW}, it holds that $y\mapsto W^{(q)}(y)$ is continuously differentiable for $y>0$.
%The value of the scale function and its derivative at zero can be derived from the Laplace transform \eqref{eq:defW} to be equal to
%\begin{equation}\label{eq:Wq0}
%W^{(q)}(0)=1/\mathtt d\ \text{and}\ W^{(q)\prime}(0^+)=(q + \nu(-\infty,0))/\mathtt d^2,
%\end{equation}
%if $X$ has bounded variation, and $W^{(q)}(0)=W(0)=0$ if $X$ has unbounded variation; see e.g.\ \cite[Exc.\ 8.5 and Lemma 8.3]{Kbook}. Moreover, if $\s>0$ it holds that $W^{(q)}\in C^\infty(0,\infty)$ with $W^{(q)\prime}(0^+)=2/\s^2$; if $X$ has unbounded variation with $\s=0$, it holds that $W^{(q)\prime}(0^+)=\infty$ (see \cite[Lemma 4]{P} and \cite[Lemma 1]{PE}).

Closely related to $W^{(q)}$ is the function $Z^{(q)}$ given by
$$
Z^{(q)}(y) = 1 + q\int_0^yW^{(q)}(z)\td z.
$$
The name ``$q$-scale function'' for $W^{(q)}$ and $Z^{(q)}$ is justified as these functions are harmonic for the process $\hat{X}$ killed upon entering $(-\infty,0)$. Here we give a few examples of scale functions. For a large number of examples of scale functions see e.g.\ Chaumont et al.~\cite{KyprChaum}, Hubalek and Kyprianou~\cite{KyprHunbalek}, Kyprianou and Rivero~\cite{KyprRivero}.

\begin{Ex}\label{Brownianwithdrift}
If $X(t) = \sigma B(t) - \mu t$ is a Brownian motion with drift $\mu$ (a standard model for small service requirements) then
$$
W^{(q)}(x) = \frac{1}{\sigma^2\delta} [\te{(-\omega+\delta) x} -
\te{-(\omega+\delta)x}],
$$
where $\delta = \sigma^{-2}\sqrt{\mu^2 + 2q\sigma^2}$ and $\omega = \mu/\sigma^2$.
\end{Ex}

\begin{Ex}\label{compound}
Suppose
$$
X(t)= \sum_{i=1}^{N(t)}\sigma_i-pt,
$$
where $p$ is the speed of the server and $\{\sigma_i\}$ are i.i.d.\  service times that are coming according to the Poisson process $N(t)$ with intensity $\lambda$. We assume that all $\sigma_i$ are exponentially distributed with mean $1/\mu$.  Then $\psi(\theta)=p\theta-\lambda \theta/(\mu+\theta)$ and the scale function of the dual $W^{(q)}$ is given by
$$
W^{(q)}(x) = p^{-1}\left(A_+ \te{q^+(q)x} - A_- \te{q^-(q)x}\right),
$$
where $A_\pm = \frac{\mu + q^\pm(q)}{q^+(q)-q^-(q)}$ with
$q^+(q)=\Phi(q)$ and $q^-(q)$ is the smallest root of $\psi(\theta)=q$:
$$
q^\pm(q)=\frac{{q} + \lambda -\mu p \pm \sqrt{({q} + \lambda-\mu p)^2 + 4 p {q}\mu}}{2 p}.
$$
\end{Ex}

\subsection{Fluctuation identities}\label{ss:identities}

The functions $W^{(q)}$ and $Z^{(q)}$ play a key role in the fluctuation theory of reflected processes as shown by the following identity (see Bertoin~\cite[Theorem VII.4 on p.\ 191 and (3) on p.\ 192]{bert96} or Kyprianou and Palmowski~\cite[Theorem 5]{KypPal}).

\begin{lem}\label{exitident}
For $\alpha >0$,
$$
E\left( \te{-\alpha \overline{X}(e_{q})}\right) =\frac{q(\alpha -\Phi \left( q\right) )}{\Phi \left( q\right) (\psi \left( \alpha \right) -q)},
$$
which is equivalent to
\begin{equation*}
P(\overline{X}(e_q)\in \td x)=\frac{q}{\Phi \left( q\right) } W^{(q)}(\td x)-qW^{(q)}(x)\td x,\qquad x>0.
\end{equation*}
Moreover, $-\underline{X}(e_q)$ follows an exponential distribution with parameter $\Phi(q)$.
\end{lem}

The scale function gives also the density $r^{(q)}(x,y)=R^{(q)}(x,\td y)/\td y$ of the $q$-potential measure
$$
R^{(q)}(x,\td y):=\int_0^\infty \te{-qt}P_x(X(t)\in \td y, \tau^-_0>t)\td t
$$
of the process $X$ killed on exiting $[0,\infty)$ when initiated from $x$. See also Pistorius~\cite{Pistoriusthesis}.

\begin{lem}\label{potentialkilled} 
Under \eqref{eq:condW}, we have that
$$
r^{(q)}(x,y)= \int_{[(x-y)^+,x]}\te{-\Phi(q) z}\left[W^{(q)\prime}(y-x+z)-\Phi(q)W^{(q)}(y-x+z)\right]\td z.
$$
\end{lem}
\begin{proof}
We start by noting that for all $x, y >0$ and $q > 0$,
$$
R^{(q)}(x, \td y)=  \frac{1}{q} P_x(X(e_q)\in \td y,\underline{X}(e_q)\geqslant 0).
$$
From the Wiener-Hopf factorisation of the L\'evy process we have that $X(e_q)-\underline{X}(e_q)$ is independent of $\underline{X}(e_q)$. This leads to
\begin{align*}
R^{(q)}(x, \td y) &= \frac{1}{q}P(X(e_q)\in \td y - x,\underline{X}(e_q)\geqslant -x) \\
               &= \frac{1}{q} P((X(e_q)-\underline{X}(e_q)) + \underline{X}(e_q)\in \td y - x,-\underline{X}(e_q)\leqslant x) \\
               &= \frac{1}{q} \int_{[(x-y)^+,x]} P(-\underline{X}(e_q)\in \td z) P(X(e_q)-\underline{X}(e_q)\in \td y - x + z).
\end{align*}
In the above expression, we integrate over the position of $-\underline{X}(e_q)$ which is not less than $0$ under $P_x$ (this leads to the condition that $ -\underline{X}(e_q)\leqslant x$ under $P=P_0$) and it is less than $X(e_q)=y$ under $P_x$ (hence, $ -\underline{X}(e_q)>x-y$ under $P$). Note that we always have that $ -\underline{X}(e_q)\geqslant 0$ under $P$, and thus the above integral is equal to the integral over $[0,x]$ when $y>x$.

Recall, that by duality $X(e_q)-\underline{X}(e_q)$ is equal in distribution to $\overline{X}(e_q)$ which has been identified in
Lemma~\ref{exitident}. In addition, the law of $-\underline{X}(e_q)$ is exponentially distributed with parameter $\Phi(q)$. We may, therefore, rewrite the expression for $R^{(q)}(x, \td y)$ as follows:
\begin{equation}\label{Rqgeneral}
R^{(q)}(x, \td y) =\int_{[(x-y)^+,x]}\te{-\Phi(q) z}\left[ W^{(q)}(\td y-x+z)-\Phi(q)W^{(q)}(y-x+z)\td y\right]\td z.
\end{equation}
Under Condition \eqref{eq:condW}, $W^{(q)}$ is differentiable and hence the last equality completes the proof.
\end{proof}

\begin{remark}
Lemma (\ref{potentialkilled}) and similar results can be proven without the assumption made in \eqref{eq:condW}, but at the cost of more complex expressions. We would have to use \eqref{Rqgeneral} instead of the much nicer form for $r^{(q)}(x,y)\td y$.
\end{remark}

\section{Equilibrium distribution of the embedded process}\label{s:embedded}
We consider the workload process at the embedded epochs $e^{(1)}_q+\cdots+e^{(n)}_q$, just after the additional service arrives. Note that this process is a Markov chain $\{Z_n, n\in \mathbf{N}\}$ with transition kernel:
\begin{equation}\label{kernel}
k(x, \td y)=P_x(F(Y(e_q))\in \td y)=\int P_x(f(Y(e_q))\in \td y) \td P^F(f),
\end{equation}
where $P^F$ is the law of $F$.

\begin{lem}\label{transitionreflected}
We have that
$P_x(Y(e_q)\in \td y)=h(x,y)\td y+\te{-\Phi(q)x}W^{(q)}(0)\delta_0(\td y)$, where
\begin{equation}
h(x,y)=qr^{(q)}(x,y)+ \te{-\Phi(q)x}\left[\frac{q}{\Phi
\left(q\right) } W^{(q)\prime}(y)-qW^{(q)}(y)\right],\label{maindensity}
\end{equation}
and where the first increment $qr^{(q)}(x,y)$ is given in
Lemma~\ref{potentialkilled}.
\end{lem}
\begin{proof}

Define $\kappa_0=\inf\{t\geqslant 0: Y(t)=0\}$, and observe that
\begin{align*}
P_x(Y(e_q)\in \td y)&=P_x(Y(e_q)\in \td y, \kappa_0>e_q)+P_x(Y(e_q)\in \td y, \kappa_0<e_q)\\
&= P_x(X(e_q)\in \td y, \tau_0^->e_q)+ P(Y(e_q)\in \td y)P_x(\tau_0^-<e_q)\\
&=qr^{(q)}(x,y)\td y+ P(\overline{X}(e_q)\in \td y)P(-\underline{X}(e_q)>x)\\
&=qr^{(q)}(x,y)\td y+ \te{-\Phi(q)x}\left[\frac{q}{\Phi \left(q\right)
} W^{(q)\prime}(y)-qW^{(q)}(y)\right]\mathbf{1}_{\{y>0\}}\td y\\&\quad+\te{-\Phi(q)x}W^{(q)}(0)\delta_0(\td y),
\end{align*}
where in the second equality we use the lack of memory of the exponential distribution and the fact that $X$ is spectrally positive; hence it crosses $0$ in continuous way. The last equality follows from Lemma \ref{exitident}.
\end{proof}

\begin{lem}
Assume that there exists an a.s.\ finite r.v.\ $F_0$ such that
\begin{equation}\label{statassump2}
F(y)\leqslant F_0\qquad\mbox{for any $y$.}
\end{equation}
Then the stationary distribution $\pi(y)$ of $Z_n$ exists and satisfies the following balance equation:
\begin{equation}\label{statembedd}
\int_{0}^\infty g(y)\td\pi(y)+g(0)\pi(0)=\int_{[0,\infty)}\int_{[0,\infty)}g(y) k(x,\td y)\td\pi(x)
\end{equation}
for any bounded function $g$.
\end{lem}

\begin{proof}
We show that  $\{Z_n, n\in \mathbf{N}\}$ is Harris ergodic, see for example
Asmussen~\cite[Theorems VII.3.3 and VII.3.5, p.\
200-201]{asmussen-APQ}. Fix $N>0$ such that $P(F_0\leqslant N)>0$ and define $\tau_N = \inf\{ n\geqslant 1: Z_n \leqslant N\}$. It is easy to construct i.i.d.\ random variables $\{F_{0,n}, n\in \mathbf{N}\}$ such that $F_{0,1} \eqdistr F_0$ and $Z_n \leqslant F_{0,n}, n\in \mathbf{N}$. Observe that
\[
P(\tau_N>k \mid Z_0 = x) = P(Z_1>N,\ldots,Z_k>N \mid Z_0=x) \leqslant P(F_{0,1}>N,\ldots,F_{0,k}>N) = P(F_0>N)^k,
\]
which implies that
\begin{equation}
\label{firstconditionfoss}
\sup_{x\geqslant 0} E[\tau_N \mid Z_0=x] <\infty.
\end{equation}
This implies Harris ergodicity, once we find a constant $p>0$ and a probability measure $Q(\cdot)$ such that
\begin{equation}
\label{secondonditionfoss}
P(Z_1 \in B \mid Z_0 = x) \geqslant p Q(B), \hspace{1cm} x\in [0,N].
\end{equation}
Let $x\in [0,N]$. We construct $p$ and $Q(\cdot)$ as follows. Recall that $\kappa_0=\inf\{t\geqslant 0: Y(t)=0\}$, observe that
\begin{align*}
P(Z_1 \in B \mid Z_0=x) &= P_x(F(Y(e_q)) \in B ) \\
&\geqslant P_x(F(Y(e_q)) \in B , \kappa_0 < e_q )\\
&=P(F(Y(e_q)) \in B) P_x(\kappa_0 < e_q )\\
&\geqslant P(F(Y(e_q)) \in B) P_N(\kappa_0 < e_q )\\
&:= Q(B)p.
\end{align*}
Since the paths of $Y(\cdot)$ are non-monotone, we have that $p>0$, implying \eqref{secondonditionfoss}.

\iffalse

By the strong Markov property of $X$ and the lack of memory of the exponential law,
for all $x\geqslant 0$ we have that
\begin{align*}
E_x P_{Z_1}(Z_2=0)&\geqslant E_x P_{Z_1}(\underline{X}(e_q)<0)P_0(Z_1=0)\\
&\geqslant E P_{F_0}(\underline{X}(e_q)<0)P_0(Z_1=0)>0.
\end{align*}
Thus $x_0=0$ is recurrent and hence $Z_n$ is a Harris recurrent Markov chain
having stationary distribution (see
Asmussen~\cite[Theorems VII.3.3 and VII.3.5, p.\
200-201]{asmussen-APQ}).
\fi

\end{proof}

\begin{remark}
The assumption (\ref{statassump2}) is satisfied when $F(y)=(B-y)^+$ with $F_0=B$. Moreover, if the functional is given by $F(y)=(B-y)^+$ and the r.v.\ $B$ has a density then we have that
\begin{equation}\label{pizero}
\pi(0)= \int_{[0,\infty)} \int_0^\infty \int_t^\infty
\left\{qr^{(q)}(x,y)\td y+ \te{-\Phi(q)x}\left[\frac{q}{\Phi
\left(q\right) } W^{(q)\prime}(y)-qW^{(q)}(y)\right]\td y\right\}
\td F_B(t)\td \pi(x),
\end{equation}
and for any bounded function $g$,
\begin{multline*}
\int_{0}^\infty g(y)\td \pi(y)= \\
\int_{[0,\infty)} \int_0^\infty \int_0^t \td y g(t-y)
\left\{qr^{(q)}(x,y)+ \te{-\Phi(q)x}\left[\frac{q}{\Phi
\left(q\right) } W^{(q)\prime}(y)-qW^{(q)}(y)\right]\right\}
\td F_B(t)\td \pi(x)
\\+W^{(q)}(0)\int_{[0,\infty)}\te{-\Phi(q)x} \td \pi(x)\int_0^\infty g(t)
\td F_B(t),
\end{multline*}
where $F_B$ is the distribution function of the generic r.v.\ $B$.
\end{remark}
In Section~\ref{s:examples} we analyse more specific examples.

\section{Steady-state workload distribution}\label{s:arbitrary}

\begin{theorem}\label{main} Suppose that the stationary distribution $\pi$ exists. Then the stationary distribution $V(\infty)$ also exists. Moreover, for a bounded function $g$,
\begin{multline*}
Eg(V(\infty))=\int_{[0,\infty)}\pi(\td x)\int_0^\infty g(y)\td y\left\{qr^{(q)}(x,y)+ \te{-\Phi(q)x}\left[\frac{q}{\Phi \left(q\right) }
W^{(q)\prime}(y)-qW^{(q)}(y)\right]\right\}\\
+g(0)W^{(q)}(0)\int_0^\infty \te{-\Phi(q) x}\pi(\td x),
\end{multline*}
where the distribution $\pi$ satisfies \eqref{statembedd} with $k$ defined via \eqref{kernel} and Lemma \ref{transitionreflected} and where $r^{(q)}(x,y)$ is given in Lemma~\ref{potentialkilled}.
\end{theorem}
\begin{proof}
The first part of the theorem follows from Asmussen~\cite[Theorem VII.6.4, p.\ 216]{asmussen-APQ}. Using the Palm inversion formula (see Asmussen~\cite[Theorem VII.6.4, p.\ 216]{asmussen-APQ}) one can identify the steady-state workload distribution $V(\infty)$ by the following identity:
\begin{equation}
Eg(V(\infty))=q\int_{[0,\infty)} \pi(\td x)E_x\int_0^{e_q}g(V(s))\td s.
\end{equation}
The RHS can be further developed as follows:
\begin{align*}
q\int_{[0,\infty)}& \pi(\td x)E_x\int_0^{e_q}g(V(s))\td s=q\int_{[0,\infty)} \pi(\td x)E\int_0^{e_q}E_x[g(Y(s))]\td s \\
&=q\int_{[0,\infty)} \pi(\td x)\int_0^\infty \te{-qt}E_x[g(Y(t))]\td t \\
&=\int_{[0,\infty)} \pi(\td x)E_x[g(Y(e_q))] \\
&=\int_{[0,\infty)} \pi(\td x)\int_0^\infty g(y)\td y
\left\{qr^{(q)}(x,y)+ \te{-\Phi(q)x}\left[\frac{q}{\Phi
\left(q\right) } W^{(q)\prime}(y)-qW^{(q)}(y)\right]\right\}\\
&+g(0)W^{(q)}(0)\int_0^\infty \te{-\Phi(q) x} \pi(\td x).
\end{align*}
\end{proof}

\section{Computational examples }\label{s:examples}

%In both examples assumption \eqref{statassump2} is satisfied,  and thus the steady state distribution $V(\infty)$ %exists.
We now turn to analysing a few specific examples. We find that there are several solution strategies. One can either solve the equations given in Sections \ref{s:embedded} and \ref{s:arbitrary} directly, or one can also take a less direct route, using Laplace transforms. We shall consider examples of both strategies.

To this end, we start with the following simple, but very useful observation. Using PASTA, $V$ has the same distribution as $U$ which is the equilibrium distribution of the Markov chain $\{U_n, n\in \mathbf{N}\}$ of the workload process embedded at times $(e^{(1)}_q+\cdots+e^{(n)}_q)-$ (i.e.\ right before the ``correction''). To obtain our main result, we first state and prove the following useful lemma.

\begin{lem}\label{taillemma1}
The following equality holds in distribution:
\begin{equation}
U \eqdistr \max \{F(U) + X(e_q), \overline X (e_q)\}.
\end{equation}
\end{lem}

\begin{proof}
If $U_0$ is $x$, we see that $U_1 \eqdistr Y(e_q)$, with $Y(0)=F(x)$. If $Y(0)=F(x)$, we further have
\[
Y(t) \eqdistr \max \{ F(x)+ X(t), \overline X(t)\},
\]
which follows, for example, by mimicking the proof of Kyprianou~\cite[Lemma 3.5, p.\ 74]{Kbook}, starting from the expression for $Y(t)$ given in Kyprianou~\cite[p.\ 19]{Kbook}. Combining these two observations leads to the statement of the lemma.
\end{proof}

\begin{Ex}\label{eg:3}
The most trivial example is when $F(y)=B\geqslant 0$. In this simple case, there is no need to use the formula derived for the generator $k(x,y)$. Using the above lemma, we see that
\begin{align*}
V(\infty) &\eqdistr \max \{B + X(e_q), \overline X (e_q)\} \\
&= \overline X (e_q)+ \max \{ B+X(e_q)- \overline X (e_q), 0\}\\
&\eqdistr \overline X (e_q) +  \max \{ B-e_{\Phi(q)},0\}.
\end{align*}
In the last equation, which follows from the Wiener-Hopf factorisation, $e_{\Phi(q)}$ is a random variable which is exponentially distributed with rate $\Phi(q)$, which is independent of everything else.
Observe that

\begin{align*}
E[\te{-s\max \{ B-e_{\Phi(q)},0\}}] &= P(e_{\Phi(q)}>B) + E[\te{-s( B-e_{\Phi(q)})};e_{\Phi(q)}<B]\\
&= P(e_{\Phi(q)}>B) + E[\te{-s( B-e_{\Phi(q)})}]-E[\te{-s( B-e_{\Phi(q)})};e_{\Phi(q)}\geqslant B] \\
&= E[\te{-\Phi(q)B}] + E[\te{-sB}]\frac{\Phi(q)}{\Phi(q)-s} -  \frac{\Phi(q)}{\Phi(q)-s}E[\te{-\Phi(q)B}]\\
&= \frac{\Phi(q)E[\te{-sB}] - s E[\te{-\Phi(q)B}]}{\Phi(q)-s}
\end{align*}
Combining this with Lemma~\ref{exitident}, we obtain
\begin{equation}
E[\te{-sV(\infty)}] = \frac{q(s -\Phi \left( q\right) )}{\Phi \left( q\right) (\psi \left( s \right) -q)}
 \frac{\Phi(q)E[\te{-sB}] - s E[\te{-\Phi(q)B}]}{\Phi(q)-s}.
 \end{equation}
This is an extension of various results in the literature focusing on clearing models (i.e.\ systems with workload removal) where $B=0$. See for example \cite{boxmaperry,Kellatotalclearing} and references therein.

\iffalse
Then $\pi(\td x)=F_B(\td x)$
and the Laplace transform of $V(\infty)$ equals
\begin{align*}
E\te{-\alpha V(\infty)}&=\int_0^\infty \td F_B(t) \int_0^\infty \td y
\te{-\alpha y}\left\{qr^{(q)}(t,y)+
\te{-\Phi(q)t}\left[\frac{q}{\Phi
\left(q\right) } W^{(q)\prime}(y)-qW^{(q)}(y)\right]\right\}\\
&+W^{(q)}(0)\int_0^\infty \te{-\Phi(q)t}\td F_B(t)\\
&=q\int_0^\infty \td F_B(t)\int_0^\infty \td y \te{-\alpha
y}r^{(q)}(t,y)+ \frac{q\tilde{F}_B(\Phi(q))(\alpha-\Phi(q))
}{\Phi(q)(\psi(\alpha)-q)},
\end{align*}
where $\tilde{F}_B(\theta)=\int_0^\infty \te{-\theta t}\td F_B(t)$ and $W^{(q)}(0)=0$. Now,
by taking $X(t)=B(t)-\m t$ (see Example~\ref{Brownianwithdrift})
we have that the Laplace exponent of the dual is equal to
$\psi(\theta)=\frac{1}{2}\theta^2+\m \theta$ and hence
$\Phi(q)=-\omega +\delta$. Then straightforward calculations give:
\begin{eqnarray*}
\lefteqn{E \te{-\alpha
V(\infty)}=\frac{q}{\sigma^2\delta}\left(\frac{1}{\alpha+\omega-\delta}-1\right)\tilde{F}_B(\delta-\omega)}\\&&
+\frac{q}{\sigma^2\delta}\left(\frac{1}{\alpha+\omega+\delta}-\frac{1}{\alpha+\omega+\delta}\right)\tilde{F}_B(\alpha)\\
&&+\frac{q}{\sigma^2\delta}\frac{1}{\alpha+\omega+\delta}\tilde{F}_B(\alpha+2\delta)
+\frac{q\tilde{F}_B(\Phi(q))(\alpha-\Phi(q))
}{\Phi(q)(\psi(\alpha)-q)}.
\end{eqnarray*}
\fi
\end{Ex}

\begin{Ex}
We now consider an example where it seems more natural to solve the equations developed in Sections \ref{s:embedded} and \ref{s:arbitrary} directly. Consider the case when $F(y)=(B-y)^+$ with $B$ being exponentially distributed with intensity $\beta$. Moreover, $X(t)= \sum_{i=1}^{N(t)}\sigma_i-pt$ is a compound Poisson process with exponentially distributed service times $\sigma_i$ with intensity $\mu$ (see also the setup of Example~\ref{compound}). Note that
$$
\psi(\theta)=p\theta-\lambda\int_0^\infty (1-\te{-\theta z})\mu \te{-\mu z}\td z=p\theta-\lambda\frac{\theta}{\mu(\mu+\theta)}
$$
and recall that $\Phi(q)=q^+(q)$. Thus,
\begin{align*}
E\te{-\alpha V(\infty)}&=q\int_{[0,\infty)}\pi(\td x)\int_0^\infty\td y \te{-\alpha y}r^{(q)}(x,y)+ \frac{q\tilde{\pi}(\Phi(q))(\alpha-\Phi(q))}{\Phi(q)(\psi(\alpha)-q)}+\frac{1}{p}(A_+-A_-)\tilde{\pi}(\Phi(q))\\
&=H(\alpha, \pi)+\frac{q\tilde{\pi}(\Phi(q))(\alpha-\Phi(q))}{\Phi(q)(\psi(\alpha)-q)}+\frac{1}{p}(A_+-A_-)\tilde{\pi}(\Phi(q)),
\end{align*}
where
\begin{multline*}
H(\theta,u)= \frac{q}{p}A^-\left\{\tilde{u}(q^+(q))\frac{q^+(q)-q^-(q)}{(\theta-q^+(q))(\theta-q^-(q))}- \tilde{u}(\theta)\frac{2q^+(q)}{\theta^2-q^+(q)^2}\right.\\
\left.+\tilde{u}(\theta+q^+(q)-q^-(q))\frac{2q^-(q)}{\theta^2-q^-(q)^2}\right\}
\end{multline*}
and
$\tilde{u}(\theta)=\int_{[0,\infty)} \te{-\theta
x}u(\td x)$. To complete the computations we have to find the LST $\tilde{\pi}$ of the stationary distribution $\pi$. By the memoryless property of the exponential distribution of $B$ we have that $\pi(\td x)=\beta \te{-\beta x}\td x$, $x\geqslant 0$. Hence,
$$
\tilde{\pi}(\theta)=\pi(0)+\frac{\beta}{\beta+\theta}.
$$
We will find now $\pi(0)$ using \eqref{pizero}.
\begin{multline*}
\pi(0)=\pi(0)\beta \int_0^\infty \te{-\beta t}\td t \int_t^\infty
 \left[\frac{q}{\Phi(q)}W^{(q)\prime}(y)-qW^{(q)}(y)\right]\td y+\\
+\beta^2\int_0^\infty \te{-\beta x} \td x \int_0^\infty \te{-\beta t}\td t
\int_t^\infty
\left\{qr^{(q)}(x,y)+ \te{-\Phi(q)x}\left[\frac{q}{\Phi(q)}W^{(q)\prime}(y)-qW^{(q)}(y)\right]\right\}\td y.
\end{multline*}
Let now $\epsilon_\beta(\td x)=\beta \te{-\beta x}\td x$ and
$$
G(q)=\frac{\beta q}{p(\beta-q^-(q))}\frac{q^-(q)-q^+(q)}{q^-(q)q^+(q)}.
$$
Then,
$$
\pi(0)=\frac{G(q)\frac{\beta}{\beta+q^+(q)}+H(0,\epsilon_\beta)- H(\beta,\epsilon_\beta)}{1-G(q)}.
$$
More general cases can be handled at the cost of more cumbersome computations. For example, we can add a compound Poisson process with phase-type jumps to the L\'evy process, and we can allow $B$ to have a phase-type distribution. See \cite{vlasiouboxma} for similar computations in a discrete-time setting.
\end{Ex}

If one cannot expect to obtain distributions in closed form, one can still aim to obtain Laplace transforms. Using similar arguments as in Example~\ref{eg:3}, we obtain the key equation (abbreviating $V=V(\infty)$)
\begin{equation}
\label{keyequation}
E[\te{-sV}] = \frac{q(s -\Phi \left( q\right) )}{\Phi \left( q\right) (\psi \left( s \right) -q)}
 \frac{\Phi(q)E[\te{-sF(V)}] - s E[\te{-\Phi(q)F(V)}]}{\Phi(q)-s}.
 \end{equation}

This equation is, of course, too complicated to solve for an arbitrary $F$, but nevertheless seems useful.

\begin{Ex}
Suppose that $F(x)=\delta x$, $\delta \in (0,1)$. This case is a generalisation of a model for the throughput behaviour of a data connection under the Transmission Control Protocol (TCP) where typically the L\'{e}vy process is a simple deterministic drift; see for example \cite{altman,guillemin,krishanu,krishanu2} and references therein.

Equation~\eqref{keyequation} reduces to
\begin{equation}
\label{keyequation2}
E[\te{-sV}] = \frac{q}{q-\psi(s)}E[\te{-s\delta V}] + \frac{qs}{\Phi(q)(\psi \left( s \right) -q)}E[\te{-\Phi(q)\delta V}].
 \end{equation}
This is an equation of the form $v(s)=g(s)v(\delta s)+h(s)$, which, since $v(0)=1$, has as (formal) solution
\[
v(s) = \prod_{j=0}^\infty g(\delta^j s) + \sum_{k=0}^\infty h(\delta^k s) \prod_{j=0}^{k-1} g(\delta^js).
\]
Specialising to our situation we obtain
\begin{eqnarray*}
v(s) &=& \prod_{j=0}^\infty \frac{q}{q-\psi(\delta^j s)} + v(\delta \Phi(q))
\sum_{k=0}^\infty \frac{qs\delta^k }{\Phi(q)(\psi \left( s\delta^k \right) -q)} \prod_{j=0}^{k-1} \frac{q}{q-\psi(\delta^j s)}\\
&=& \prod_{j=0}^\infty \frac{q}{q-\psi(\delta^j s)} + v(\delta \Phi(q)) \frac{s}{\Phi(q)}
\sum_{k=0}^\infty \delta^k  \prod_{j=0}^{k} \frac{q}{q-\psi(\delta^j s)}.
\end{eqnarray*}
Since $\psi(0)=0$, it easily follows that the infinite products converge, and the final expression for $v(s)$ yields an equation from which the only remaining unknown constant $v(\delta \Phi(q))$ can be solved explicitly.
\end{Ex}

\section{Tail behaviour}\label{tail}

In this section we consider the tail behaviour of $V=V(\infty)$, under a variety of assumptions on the tail behaviour of the
%assuming that
the L\'{e}vy measure $\nu$.
% is subexponential, or, more generally, convolution equivalent.
%We also analyze the Cram\'{e}r case when L\'{e}vy measure $\nu$ is light-tailed.
The treatment in this section can be seen as the continuous-time analogue of the results in
%a similar discrete-time result in
Vlasiou and Palmowski \cite{vlasiou08}. There, a similar result is shown where $e_q$ is geometrically distributed, $X(\cdot)$ is replaced by a (general) random walk, and $F(y)=(B-y)^+$, with $B$ identical to the increments of the random walk. In \cite{vlasiou08} we modify ideas from Goldie~\cite{Goldie91} in the light-tailed case and develop stochastic lower and upper bounds in the heavy-tailed case. Here we  take a different approach which is based on the well-developed fluctuation theory of spectrally one-sided L\'{e}vy processes. Before we present our main results, we first state some lemmas.

\begin{lem}\label{tailbounds1}
The following (in)equalities hold:
\begin{align}
P(U>x)&= P(X(e_q)+F(U)>x) \nonumber\\&+\int_0^\infty (P(\overline X(e_q)>x)-P(\overline X(e_q)>x+y)) P(-\underline{X}(e_q)-F(U)\in \td y),\label{bound1}\\
P(U>x)&\leqslant P(X(e_q)+F(U)>x)+P(\overline X(e_q)>x)P(-\underline{X}(e_q)>F(U)),\label{bound3}\\
P(U>x)&\geqslant P(\overline X(e_q)>x)P(-\underline{X}(e_q)>F(U)).\label{bound4}
\end{align}
\end{lem}

\begin{proof}
All identities follow from  Lemma \ref{taillemma1}, the decomposition $X(e_q)=\overline X(e_q)- (X(e_q)-\overline X(e_q))$, and the Wiener-Hopf factorisation factorisation which implies that $\overline X(e_q)$ and $\overline X(e_q)-X(e_q)$ are independent. We also use the simple fact that $\overline X(e_q)-X(e_q)$ has the same law as
$-\underline{X}(e_q)$ (see Kyprianou~\cite[p.\ 158]{Kbook}).
\end{proof}

%The crucial observation comes from the Wiener-Hopf factorization.
In addition, we need two standard results from the literature on L\'evy processes.

\begin{lem}\label{ladderrepr} (Kyprianou~\cite[p.\ 165]{Kbook})
The random variable $\overline X(e_q)$ has the same law as $H(e_{\kappa(q,0)})$, where $\{(L^{-1}(t), H(t)), t\geqslant 0\}$ is a ladder height process
of $X$ with the Laplace exponent $\kappa (\varrho,\zeta)$ defined by
$$
E\te{\varrho L^{-1}(t)+\zeta H(t)}=\te{\kappa (\varrho,\zeta)t}.
$$
\end{lem}

From the above, one can easily derive the following version of the Pollaczek-Khinchine formula.
\begin{lem} \label{pollkhintchine}(Bertoin~\cite[p.\ 172]{bert96})
The following identity holds:
$$
P(\overline X(e_q)>x)=\kappa(q,0)U^{(q)}(x,\infty),
$$
where
$$
U^{(q)}(\td x)=\int_0^\infty\int_0^\infty \te{-qs}P(H(t)\in \td x, L^{-1}(t)\in \td s)\td t
$$
is the renewal function of the ladder height process $\{(L^{-1}(t),H(t)), t<L(e_q)\}$ and $L(\cdot)$ is a local time of $X$.
\end{lem}

We now turn to the tail behaviour of $U$. Let $\Pi(A)=\nu(-A)$ be the L\'{e}vy measure of the spectrally positive L\'{e}vy process $X$ (with support on $R_+$). First we investigate the case where the   L\'{e}vy measure is a member of the so-called convolution equivalent class $\mathcal{S}^{(\alpha)}$. To define this class, take $\alpha \geqslant 0$. We shall say that measure $\Pi$ is {\it convolution equivalent} ($\Pi \in \mathcal{S}^{(\alpha)}$) if for fixed $y$ we have that
\begin{align*}
\lim_{u\to\infty}\frac{\bar \Pi(u-y)}{\bar \Pi(u)}&=\te{\alpha y},       &&\mbox{if $\Pi$ is nonlattice,}\\
\lim_{n\to\infty}\frac{\bar \Pi(n-1)}{\bar \Pi(n)}&=\te{\alpha},         &&\mbox{if $\Pi$ is lattice with span $1$,}
\end{align*}
and
$$
\lim_{u\to\infty}\frac{\bar \Pi^{*2}(u)}{\bar \Pi(u)}=2\int_0^\infty \te{\alpha y} \Pi(\td y),
$$
where $*$ denotes convolution and $\bar \Pi(u)=\Pi((u,\infty))$. When $\alpha=0$, then we are in the subclass of subexponential measures and there is no need to distinguish between the lattice and non-lattice cases (see \cite{bertdoneytrans}).
%It is heavy-tailed case. The case of $\alpha >0$ is an intermediate case.
We start from the following auxiliary result, which is the continuous-time analogue of Lemma 2 in \cite{vlasiou08}.

\begin{lem}
\label{taillemma2}
Assume that $\Pi\in \mathcal{S}^{(\alpha)}$ and $\psi(\alpha)<q$ for $\psi(\alpha) =\log E \te{\alpha X(1)}$.
Then
\begin{align}
P(X(e_q)>x) &\sim \frac{q}{(q-\psi(\alpha))^2}\bar \Pi(x),\label{xasalpha}\\
P(\overline X(e_q)>x) &\sim \frac{q}{(q-\psi(\alpha))^2}\frac{\Phi(q)+\alpha}{\Phi(q)}\bar \Pi(x),\label{barxasalpha}
\end{align}
where $f(x)\sim g(x)$ means that
$\lim_{x\to\infty}f(x)/g(x)=1$.
\end{lem}
\begin{remark}
Note that for $\alpha=0$
$$P(\overline X(e_q)>x)\sim P(X(e_q)>x) \sim \frac{1}{q}\bar \Pi(x).$$
\end{remark}
\begin{proof}
It is well known that $P(X(t)>x) \sim E \te{\alpha X(t)}\bar \Pi_{X(t)}(x)$ for $t$ fixed as $x\rightarrow\infty$, where
%$m_\Upsilon(\theta)=\int_0^\infty e^{\theta y} d\Upsilon (y)$ and
$\Pi_{X(t)}$ is a L\'{e}vy measure of
$X(t)$ (see Embrechts et al.\ \cite{embrechts}). Since $X(t)$ is infinitely divisible we have $\Pi_{X(t)}(\cdot)=t\Pi(\cdot)$ and hence $P(X(t)>x)\sim t(E \te{\alpha X(1)})^t\bar \Pi(x)$. Since $X(e_q) \leqslant \overline X(e_q)$ by (\ref{barxasalpha}) and the dominated convergence theorem we obtain
\eqref{xasalpha}. We will use similar arguments as in the proof of Lemma 3.5 of Kl\"{u}ppelberg et al.\ \cite{KAM}.
For $\Pi_H\in \mathcal{S}^{(\alpha)}$ note that
$$
P(H(t)>u)\sim t (E\te{\alpha H(1)})^t\bar \Pi_H(u),
$$
where $\Pi_H$ is the L\'{e}vy measure of the process $\{H(t), t< e_{\kappa(q,0)}\}$ (see Embrechts et al.\ \cite{embrechts}).
Using uniform in $u$ Kesten bounds \cite{KAM}:
$$
P(H(t)>u)\leqslant P(H([t]+1)>u)\leqslant K(\epsilon)(E\te{\alpha H(1)}+\epsilon)^{[t]+1}\bar \Pi_H(u)
$$
for any $\epsilon >0$ and some constant $K(\epsilon)$, and the dominated convergence theorem, we derive
by Lemma \ref{ladderrepr},
\begin{equation}\label{firstlimitalpha}\lim_{u\to\infty}\frac{P(\overline X(e_q)>u)}{\bar \Pi_H(u)}=
\frac{\kappa(q,0)}{(\kappa(q,0)-\log E\te{\alpha H(1)})^2}.\end{equation}
The Wiener-Hopf factorisation yields that
$$
E\te{\alpha X(e_q)}=E\te{\alpha H(e_{\kappa(q,0)})}E\te{\alpha \hat{H}(e_{\hat{\kappa}(q,0)})},
$$
where $(\hat{L}^{-1}(t), \hat{H}(t))$ is a downward ladder height process with Laplace exponent $\hat{\kappa}(\varrho,\zeta)$. Since $X$ is spectrally positive, we can choose the process $\hat{H}(t)=-t$ and hence
$$
E\te{\alpha \hat{H}(e_{\hat{\kappa}(q,0)})} =\hat{\kappa}(q,0)\int_0^\infty \te{-\hat{\kappa}(q,0)t}\te{-\alpha t}\td t=\frac{\hat{\kappa}(q,0)}{\hat{\kappa}(q,\alpha)},
$$
where $\hat{\kappa}(q,\alpha)=\Phi(q)+\alpha$.
Thus
$$
\frac{q}{q-\psi(\alpha)}\frac{\hat{\kappa}(q,\alpha)}{\hat{\kappa}(q,0)}=\frac{\kappa(q,0)}{\kappa(q,0)-\log E\te{\alpha H(1)}}.
$$
Using the well known fact that $q=\kappa(q,0)\hat{\kappa}(q,0)$ (see Kyprianou~\cite[p.\ 166]{Kbook}) we identify the right-hand side of \eqref{firstlimitalpha} as
$$
\lim_{u\to\infty}\frac{P(\overline X(e_q)>u)}{\bar \Pi_H(u)}= \frac{q}{(q-\psi(\alpha))^2}\frac{(\Phi(q)+\alpha)^2}{\Phi(q)}.
$$
Now using similar arguments like in Vigon \cite{Vigon} (see also Kyprianou~\cite[Th.\ 7.7 on p.\ 191]{Kbook}
and Kyprianou~\cite[Th.\ 7.8 on p.\ 195]{Kbook}) we derive
$$
\bar \Pi_H(u)=\int_0^\infty \bar \Pi(u+y) \hat{V}(\td y),
$$
where $\hat{V}(y)$ is the renewal function of the downward ladder height process
$\{(\hat{L}^{-1}(t), \hat{H}(t)), t<\hat{\kappa}(q,0)\}= \{(\hat{L}^{-1}(t), \hat{H}(t)), \hat{L}^{-1}(t)<e_q\}$.
Thus
\begin{align*}
\lim_{u\to\infty}\frac{\bar \Pi_H(u)}{\bar \Pi(u)}&=\int_0^\infty \te{-\alpha y} \hat{V}(\td y)\\
&=\int_0^\infty E\te{-q\hat{L}^{-1}(t)-\alpha\hat{H}(t)}\td t\\
&=\int_0^\infty \te{-\hat{\kappa}(q,\alpha)t}\td t=\frac{1}{\hat{\kappa}(q,\alpha)}=\frac{1}{\Phi(q)+\alpha}.
\end{align*}
Hence, by \cite{EmG82} also $\Pi_H\in \mathcal{S}^{(\alpha)}$ if and only if $\Pi\in \mathcal{S}^{(\alpha)}$. This completes the proof.
\end{proof}

It is known \cite{embrechts} that if for independent  random variables $\chi_i$ ($i=1,2$) we have
$P(\chi_i>u)\sim c_i\bar G(u)$ as $u\to\infty$ and $G\in \mathcal{S}^{(\alpha)}$, then
$P(\chi_1+\chi_2>u)\sim (c_1E\te{\alpha \chi_2}+c_2E\te{\alpha \chi_1})\bar G(u)$.
This observation and (\ref{bound1}) in Lemma \ref{tailbounds1} and Lemma \ref{taillemma2}
yield the following main result.

\begin{theorem}
Assume that $\Pi\in \mathcal{S}^{(\alpha)}$ and $\psi(\alpha)<q$. Moreover, let $F(y) \leqslant F_0 (\geqslant 0)$ for any $y$, and assume that there exists a constant $c\geqslant 0$ such that $P(F(y)>x) \sim P(F_0>x)\sim c \bar \Pi(x)$ as $x\rightarrow\infty$ for each $y$ (If $c=0$ then $P(F(y)>x) = o(\bar \Pi(x))$). Then
\begin{multline*}
\lefteqn{P(U>x) \sim \left\{cE\te{\alpha X(e_q)}+\frac{q}{(q-\psi(\alpha))^2}E\te{\alpha F(U)}\right.}\\
\left.+\frac{q}{(q-\psi(\alpha))^2}\frac{\Phi(q)+\alpha}{\Phi(q)}E\left[\left(1-\te{-\alpha(-\underline{X}(e_q)-F(U))}\right);
-\underline{X}(e_q)-F(U)>0\right]\right\} \bar \Pi(x)
\end{multline*}
as $x\rightarrow \infty$.
\end{theorem}

The conditions in this theorem are satisfied by both examples $F(y)=0$ (in which case we take $F_0=0,c=0$) and $F(y)=(B-y)^+$ (in which case $F_0=B$). If $\Pi$ is subexponential ($\Pi\in \mathcal{S}^{(0)}$), then
$$
P(U>x)\sim \left(c+\frac{1}{q}\right)\bar \Pi (x).
$$

We will consider now the Cram\'{e}r case (light-tailed case). Assume that there exists $\Phi(q)$ such that
\begin{equation}\label{cramer1}
\psi(\Phi(q))=q
\end{equation}
and that
\begin{equation}\label{cramer2}
m(q) := \left.\frac{\partial \kappa(q,\beta)}{\partial \beta}\right|_{\beta=-\Phi(q)}<\infty.
\end{equation}
Note that if $\Pi \in \mathcal{S}(\alpha)$ and $\psi(\alpha)<q$, then condition \eqref{cramer1} is not satisfied.
Moreover, we assume that
\begin{equation}\label{cramer3}
E\te{\Phi (q)F(U)}<\infty.
\end{equation}

\begin{theorem}
Assume that (\ref{cramer1})-(\ref{cramer3}) hold and that the support of $\Pi$ is non-lattice. Then
$$P(U>x)\sim C \te{-\Phi(q) x}$$
as $x\rightarrow \infty$, where
$$C=P(-\underline{X}(e_q)>F(U))\kappa(q,0)\left(\Phi(q)
m(q)
\right)^{-1}.$$
\end{theorem}
\begin{proof}
We  introduce the new probability measure 
$$
\left.\frac{dP^{\theta}}{dP}\right|_{\mathcal{F}_t}=\te{\theta X(t)-\psi(\theta)t},
$$
where $\mathcal{F}_t$ is a natural filtration of $X$. On $P^\theta$, the  process $X$ is again a spectrally positive L\'{e}vy process with the L\'{e}vy measure $\Pi_\theta(\td x)=\te{\theta x}\Pi(\td x)$, which is also nonlattice. Let $U_\theta^{(q)}$ be the renewal function appearing in Lemma \ref{pollkhintchine} with $P$ replaced by $P^\theta$. Recall that $L^{-1}(t)$ is a stopping time. Hence, from the optional stopping theorem, we have that
\begin{align*}
\te{-\Phi(q) x}U_{\Phi(q)}^{(q)}(\td x)&= \int_0^\infty \int_0^\infty \te{-\Phi(q)x}P^{\Phi(q)}(H(t)\in \td x, L^{-1}(t)\in \td s)\td t\\
&=\int_0^\infty \int_0^\infty \te{-\Phi(q)x}\te{-qs+\Phi(q)x}P(H(t)\in \td x, L^{-1}(t)\in \td s)\td t =U^{(q)}(\td x).
\end{align*}
We  follow now Bertoin and Doney \cite{bertdon} (see also Kyprianou~\cite[Th.\ 7.6 on p.\ 185]{Kbook}). From Lemma \ref{pollkhintchine} we have
$$
\te{\Phi(q)x}P(\overline X(e_q)>x)=\kappa(q,0)\int_x^\infty \te{-\Phi(q)(y-x)} U_{\Phi(q)}^{(q)}(\td y)
=\kappa(q,0)\int_0^\infty \te{-\Phi(q)z} U_{\Phi(q)}^{(q)}(x+\td y).
$$
From Kyprianou~\cite[Th.\ 5.4 on p.\ 114]{Kbook} it follows that $U_{\Phi(q)}^{(q)}(\td y)$ has a nonlattice support.
From the key renewal theorem (see Kyprianou~\cite[Cor.\ 5.3 on p.\ 114]{Kbook}) the measure $U_{\Phi(q)}(x+\td y)$ converges weakly
to the Lebesgue measure $\frac{1}{E^{\Phi(q)}H(1)}\td y$ (see Kyprianou~\cite[Th.\ 7.6 on p.\ 185]{Kbook}).
Thus
$$
\lim_{x\to\infty}\te{\Phi(q)x}P(\overline X(e_q)>x)=\frac{\kappa(q,0)}{\Phi(q)E^{\Phi(q)}H(1)}.
$$
Observe that
\begin{multline*}
E^{\Phi(q)}H(1)=\int_0^\infty t \te{-t}E^{\Phi(q)}H(1)\td t =\int_0^\infty \te{-t}E^{\Phi(q)}H(t)\td t= \int_0^\infty x U^{(1)}(\td x)\\
= \int_0^\infty \te{-t}\td t  \int_0^\infty x  P^{\Phi(q)}(H(t)\in \td x) =\int_0^\infty \te{-t}\int_0^\infty x \te{\Phi(q)x -qs}  P(H(t)\in \td x, L^{-1}(t)\in \td s)\td t\\
=\int_0^\infty \te{-t}E H(t)\te{\Phi(q)H(t) -qL^{-1}(t)}\td t =\int_0^\infty t\te{-t-\kappa(q,-\Phi(q))t}\td t
\left.\frac{\partial \kappa(q,\beta)}{\partial \beta}\right|_{\beta=-\Phi(q)}.
\end{multline*}
From the Wiener-Hopf factorisation (see Kyprianou~\cite[p.\ 167]{Kbook}) it follows
that
$$
q-\psi(\theta)=\kappa(q,-\theta)\hat{\kappa}(q,\theta).
$$
From the convexity of the Laplace exponents $\phi$ and $\psi$ we have that $\hat{\kappa}(q,\Phi(q))=2\Phi(q)>0$ and
hence $\kappa(q,-\Phi(q))=0$. Finally,
$$
E^{\Phi(q)}H(1)=\left.\frac{\partial \kappa(q,\beta)}{\partial \beta}\right|_{\beta=-\Phi(q)}.
$$
Note that by \eqref{cramer1} and \eqref{cramer3}, $P(X(e_q)>x)={\rm o}(\te{-\Phi(q)x})$ and $P(X(e_q)+F(U)>x)={\rm o}(\te{-\Phi(q)x})$.
Inequalities (\ref{bound3}) and (\ref{bound4}) in Lemma \ref{tailbounds1}
complete the proof.
\end{proof}

\phantomsection
\addcontentsline{toc}{section}{Acknowledgments}
\subsection*{Acknowledgments}
We thank Bert Zwart for his helpful advice and useful comments. This work is partially supported by the Ministry of Science and Higher Education of Poland under the grant N N2014079 33 (2007-2009).

%\bibliographystyle{apt}
%\bibliography{maria}

\phantomsection

\end{document}